\documentclass[12pt]{amsart}
\usepackage{graphicx}
\graphicspath{{../Figures/}}
\usepackage{amssymb}
\usepackage{bm}
\usepackage{tikz}
\usepackage{mathtools}
\allowdisplaybreaks
\def\smash#1{{\setbox0=\hbox{#1}\wd0=0pt  \box0\relax}}

\newtheorem{theorem}{Theorem}
\newtheorem{lemma}[theorem]{Lemma}
\newtheorem{proposition}[theorem]{Proposition}
\newtheorem{corollary}[theorem]{Corollary}

\theoremstyle{definition}
\newtheorem{definition}[theorem]{Definition}
\newtheorem{notation}[theorem]{Notation}

\theoremstyle{remark}
\newtheorem{remark}[theorem]{Remark}
\newtheorem{example}[theorem]{Example}

\newcommand{\ab}{\allowbreak}
\newcommand{\alg}{\textrm{alg}}

\newcommand{\bD}{{\bm D}}
\newcommand{\E}{\mathrm{E}}

\let\phi=\varphi
\newcommand{\mrp}{\textsc{mp}}

\newcommand{\tr}{\textrm{tr}}

\newcommand{\ltr}{{\reflectbox{\tiny$\Gamma$}}}

\newcommand\EE{\mathbb{E}}
\newcommand{\bC}{\mathbb{C}}

\newcommand{\cA}{\mathcal A}
\newcommand{\cB}{\mathcal{B}}
\newcommand{\cD}{\mathcal D}

\newcommand{\cN}{\mathcal N}

\newcommand{\ds}{\displaystyle}

\newcommand{\thebottomline}{%
\renewcommand{\thefootnote}{}
\renewcommand{\footnoterule}{}
\phantom{M}\footnotetext{\hfill\tiny\textit{\noindent\romannumeral\day.\romannumeral\month.\romannumeral\year}}}

\title[{Cyclic Group and the Transpose of an $R$-cyclic
    matrix}] {The Cyclic Group and \\ the Transpose of an
  $R$-cyclic matrix}

\author[arizmendi]{Octavio Arizmendi$^{(\dagger)}$}
\address{Centro de Investigaci{\'o}n en Matem{\'a}ticas,
  Guanajuato, Mexico} \email{octavius@cimat.mx}
\thanks{$^{(\dagger)}$Research supported by Conacyt Grant
  A1-S-9764 and Simons CRM Scholar-in-Residence, March
  2019.}

\author[mingo]{James A. Mingo$^{(*)}$} \address{Department
  of Mathematics and Statistics, Queen's University, Jeffery
  Hall, Kingston, Ontario, K7L 3N6, Canada}

\email{mingo@mast.queensu.ca} 

\thanks{$^{(*)}$Research supported by a Discovery Grant from
  the Natural Sciences and Engineering Research Council of
  Canada and Simons CRM Scholar-in-Residence, March 2019.}

\dedicatory{Dedicated to Dan-Virgil Voiculescu on his
  70$^{\,th}$ birthday.}

\begin{document}

\begin{abstract}
We show that using the cyclic group the transpose of an
$R$-cyclic matrix can be decomposed along diagonal parts
into a sum of parts which are freely independent over
diagonal scalar matrices. Moreover, if the $R$-cyclic matrix
is self-adjoint then the off-diagonal parts are
$R$-diagonal.
\end{abstract}

\maketitle

\section{Introduction}
After self-adjoint operators there are only a few classes of
operators with a good spectral theory; one of them is the
class of $R$-diagonal operators introduced by Nica and
Speicher. Indeed, Haagerup and Schultz \cite{hs} showed that
for $R$-diagonal operators in a von Neumann algebra which is
a factor of type II, one has an abundance of invariant
subspaces. The standard examples of $R$-diagonal operators
include Haar unitaries and circular operators. More
generally one can consider the product $ua$ where $u$ is a
Haar unitary in a $*$-probability space, $*$-free from $a$,
see \cite[Cor. 15.9]{ns}. $R$-cyclic operators were
introduced by Nica, Shlyakhtenko, and Speicher in
\cite{nss02} as a generalization of $R$-diagonal
operators. A simple way to see this is to consider the
following example from Nica-Speicher \cite[Thm. 14.18]{ns}.

Suppose $(\cA, \phi)$ is a non-commutative $*$-probability
space containing a set $\{E_{ij}\}_{i,j = 1}^d$ of matrix
units in $\cA$. Letting $\cB = E_{11}\cA E_{11}$, we may write $\cA =
M_d(\cB)$ and any $X \in \cA$ as $X = (x_{ij})_{i,j = 1}^d$
with $x_{ij} = E_{1i} X E_{j1}$. If $X$ is self-adjoint and
free from $\{E_{ij}\}_{i,j=1}^d$, then the off-diagonal
entries $x_{ij}$ ($i\not = j$) are $R$-diagonal.  Let us
recall the definition of $R$-diagonal and $R$-cyclic
operators from \cite[Lects. 15 and 20]{ns}.

\begin{definition}\label{def:r-diagonal}
If $(\cA, \phi)$ is a $*$-probability space and $a \in \cA$,
we shall write $a^{(1)}$ for $a$ and $a^{(-1)}$ for
$a^*$. Then $a$ is $R$-diagonal if the free cumulant
$\kappa_n\big(a^{(\epsilon_1)}, a^{(\epsilon_2)}, \dots,
a^{(\epsilon_n)}\big) = 0$ unless: $n$ is even and
$\epsilon_i = - \epsilon_{i+1}$ for $1 \leq i \leq n-1$.
\end{definition}

\begin{definition}\label{def:r-cyclic}
Let $(\cA, \phi)$ be a non-commutative probability space. A
matrix $A = (a_{ij})_{ij} \in M_d(\cA)$ is
$R$-\textit{cyclic} if for all $n$ and all $i_1, \dots i_n,
j_1,\ab \dots,\ab j_n \in [d]$ the free cumulant
$\kappa_n(a_{i_1j_1}, \dots, a_{i_nj_n}) = 0$ unless $j_1 =
i_2, \dots, j_n = i_1$.
\end{definition}

\begin{example}
Suppose $(\cA, \tau)$ is a $*$-probability space and $x \in
\cA$.  Let $X = \left(\begin{smallmatrix}0 & x \\ x^* &
  0\end{smallmatrix}\right)$. Then $x$ is $R$-diagonal if
  and only if $X$ is $R$-cyclic in $M_2(\cA)$ relative to
  $\phi = \tr \otimes \tau$ on $M_2(\bC) \otimes \cA =
  M_2(\cA)$. In \cite[Thm. 1.2]{nss01} it was shown that the
  $R$-diagonality of $x$ was equivalent to the freeness of
  $X$ from $M_2(\bC)$ over the subalgebra of $2 \times 2$
  diagonal scalar matrices.
\end{example}

\begin{example}\label{ex:free_from_scalars}
In \cite[Ex. 20.4 ]{ns} it is shown that if $X \in \cA$ is
free from a set of matrix units in $\cA$, then $X$ is
$R$-cyclic relative to this set of matrix units. But this is
not the most general situation. In \cite[Thm. 8.2]{nss02} it
is shown that if $(M_d(\cA), \phi)$ is a non-commutative
probability space and $A \in M_d(\cA)$ is $R$-cyclic then
$A$ is free from $M_d(\bC)$ over $\cD_d$, where $\cD_d
\subseteq M_d(\bC)$ is the subalgebra of diagonal matrices.
We shall work with this formulation. 
\end{example}

\begin{remark}\label{rem:abundance}
The significance of $R$-diagonal and $R$-cyclic matrices is
that orthogonally invariant (and hence unitarily) random
matrix models produce $R$-cyclicity and $R$-diagonality. In
\cite[Thm. 6.2]{cc} it was shown that orthogonally invariant
and constant matrices are asymptotically free provided the
constant matrices converge in distribution (another proof
was given in \cite[Thm.~36]{mp}). Suppose $X_N$ is an
orthogonally invariant ensemble with a limit
distribution. Write $N = d \times p$ and $M_N(\bC) =
M_d(\bC) \otimes M_p(\bC)$. Let $E_{ij} = e_{ij} \otimes 1_p
\in M_d(\bC) \otimes M_p(\bC)$ be the standard matrix
units. Then $\{E_{ij}\}_{ij}$ converges to a set of matrix
units as $p \rightarrow \infty$. Indeed the mixed moments
are stationary. Then $X_N$ and $\{E_{ij}\}_{ij}$ are
asymptotically free and thus we may realize the limit
distribution of $X_N$ as an element $x$ in a non-commutative
probability space $(\cA, \phi)$ where there is a system of
matrix units $\{E_{ij}\}_{ij}$ which are free from $x$. By
our earlier discussion such an $x$ is $R$-cyclic relative to
$\{E_{ij}\}_{ij}$.  The orthogonal invariance assumption is
not a necessary condition, for example constant matrices and
Wigner matrices are asymptotically free. A proof for real
Wigner matrices is given in \cite[Thm. 4.20]{ms}. Thus
$R$-cyclic operators occur very naturally. 
\end{remark}

\begin{remark}\label{rem:cyclic_group}
Our main tools will be the two representations of the cyclic
group of order $d$ in $d \times d$ matrices. The generators
for these actions will be the diagonal matrix $\bD$ with
$d^{th}$ roots of unity on the diagonal and $S$ the matrix
that cyclically permutes the standard basis elements.  See
Notation \ref{def:xs-decomposition} for the notation. These
two representations are intertwined by the Fourier transform
matrix. This marks yet another case where the action of a
symmetry group produces freeness but the first time the
group is the cyclic group. The crucial point for us is that
we can write the transpose in terms of these dual actions,
see Lemma \ref{lemma:x-relation}, and $S$ and $\bD$ are free
from our matrix $X$ over the diagonal scalar matrices.

Another place where a group invariance plays a role is in
traffic freeness. In this case the group is the symmetric
group. A recent paper of Au, C\'ebron, Dahlqvist, Gabriel,
and Male \cite{au} shows the connection to freeness over the
diagonal.
\end{remark}

\begin{remark}\label{rem:transpose}
The role of the transpose in free probability arose in the
work of Aubrun \cite{a} when he showed that in a certain
regime the partial transpose of Wishart matrix converged to
a semi-circle law. To review this, let $G_1, \dots ,
G_{d_1}$ be independent $d_2 \times p$ complex Gaussian
random matrices. By this we mean $G_i =
(g^{(i)}_{j,k})_{j,k}$ with $\{g^{(i)}_{j,k}\}_{i,j,k}$
independent complex Gaussian $\cN(0,1)$ random variables. We
let
\begin{equation}\label{eq:wishart}
W =  \frac{1}{d_1d_2}\left(
\begin{array}{c}
G_1 \\ \hline
\vdots \\ \hline
G_{d_1}
\end{array}
\right)
\left(
\begin{array}{c|c|c}
G_1^* & \cdots & G_{d_1}^*
\end{array} \right) =
\frac{1}{d_1d_2} (G_iG_j^*)_{ij}
\end{equation}
and $W^\ltr = \ds \frac{1}{d_1d_2} (G_jG^*_i)_{ij}$ be the
\textit{partial transpose} of $W$. Aubrun showed that when
$d_1, d_2$ and $p \longrightarrow \infty$ such that $\ds
\frac{p}{d_1 d_2} \rightarrow c$, $W^\ltr$ converges to a
semi-circular operator with free cumulants $\kappa_1 =
\kappa_2 = c$. If we fix $d_1$ and have $d_2, p \rightarrow
\infty$ we are in the regime of Banica and Nechita
\cite{bn}.  They showed that $d_1W^\ltr$ converges in
distribution to the free difference of two Marchenko-Pastur
laws.

Recall that for each $0 < c < \infty$ there is a probability
distribution called the Marchenko-Pastur law with parameter
$c$, denoted $\mrp_c$.  We let $a = (1 - \sqrt c)^2$ and $b
= (1 + \sqrt c)^2$; $\mrp_c$ has density $\ds\frac{ \sqrt{(b
    - t)(t - a)}}{2 \pi t}$ on the interval $[a, b]$ for $c
\geq 1$ and for $0 < c < 1$ has in addition an atom of mass
$1 - c$ at 0, see \cite[Def. 2.11]{ms}. Note that in
Eq. (\ref{eq:wishart}) we have used a different
normalization than Aubrun and Banica-Nechita, in order to
simplify the notation.

So let us suppose that $d_1$ is fixed. Then $W$ converges to
$w$ an operator with distribution $\mrp_c$. In
\cite[Thm. 3.7]{mp} the mixed moments of $w$ and $w^t$ were
given. If we adopt the convention here that $w^{(1)} = w$
and $w^{(-1)} = w^t$ then \[\phi(w^{(\epsilon_1)} \cdots
w^{(\epsilon_n)}) = \sum_{\pi \in NC(n)} c^{\#(\pi)}
d_1^{f_\epsilon(\pi)}\] where $f_\epsilon(\pi) = \#(\epsilon
\gamma \delta \gamma^{-1} \epsilon \vee \pi \delta \pi^{-1})
+ \#(\pi) - (n + 1) \leq 0 $ is an integer. See \cite[\S
  3]{mp} for an explanation of the notation. The point we
need here is that $W$ and $W^\ltr$ have a joint limit
distribution and $\phi(ww^t) = cd_1^{-1} + c^2 \not= c^2 =
\phi(w) \phi(w^t)$. In particular $w$ and $w^t$ are not
free.

Since $W$ is asymptotically free from the matrix units
$\{E_{ij}\}_{ij=1}^{d_1}$ we have that we may realize $w$ as
a matrix $(w_{ij})_{ij=1}^{d_1}$ with entries in a
non-commutative probability space and $w$ is free from the
matrix units. Thus $w$ is $R$-cyclic.

Banica and Nechita showed that when $d_1$ is fixed,
$d_1W^\ltr$ converged in distribution to an operator $x_1 -
x_2$ in a non-commutative probability space where $x_1$ and
$x_2$ are free and which have distribution $\mrp_{c_1}$
and $\mrp_{c_2}$ respectively with $c_1 = {c d_1
  \frac{d_1 + 1}{2}}$ and $c_2 = {c d_1 \frac{d_1 -
    1}{2}}$.

Note that if $x = x_ 1 - x_2$ with $x_i \in \mrp_{c_i}$
then the free cumulants $\{\kappa_n\}_n$ of $x_i$ are
$\kappa_n = c_i$ for all $n$. Thus the free cumulants of
$x$ are given by
\[
\kappa_n = c_1 + (-1)^n c_2 = c d_1 \bigg( \frac{d_1 +
  1}{2} + (-1)^n \frac{d_1 - 1}{2}\bigg) =
\begin{cases}
  c d_1 d_1 & n \mathrm{\ even\ } \\ c d_1  & n
  \mathrm{\ odd\ }
\end{cases}
\]

One of the motivations for the present work is to present an
intrinsic description of this distribution in terms of the
matrix $W^\ltr$ itself.  Indeed we shall show that the
diagonal decomposition of $d_1W^\ltr$ converges to a free
family of $d_1/2 +1$ self-adjoint operators (assuming $d_1$
is even) such that all even cumulants are $cd_1$ and all odd
cumulants are $0$ except for the first operator which has
all cumulants equal to $cd_1$.

Let us illustrate this when $d_1 = 2$. We write the limit
distribution of $W$ as $w = \frac{1}{d_1}\begin{pmatrix}
  w_{11} & w_{12} \\ w_{21} & w_{22} \end{pmatrix}$. We have
that $w$ is $\mrp_c$ and free from $M_2(\bC)$. Then the
limit distribution of $W^\ltr$ is $w^t =
\frac{1}{d_1}\begin{pmatrix} w_{11} & w_{21} \\ w_{12} &
  w_{22} \end{pmatrix}$. We write $d_1w^t = X_0 + X_1$ with
\[X_0 = \begin{pmatrix} w_{11} & 0 \\ 0 &
  w_{22} \end{pmatrix} \mbox{ and\ } X_1 = \begin{pmatrix} 0 &
  w_{21} \\ w_{12} & 0 \end{pmatrix}.\] In
\cite[Thm. 6.14]{mp} it was shown that $X_0$ and $X_1$ are
free. $X_0$ is $\mrp_{d_1c}$ and $X_1$ is an even operator
with even cumulants $d_1 c$.  In this paper we extend this
to the general case $d_1 \geq 1$. A precise statement is
given in Section \ref{sec:return_wishart}, Theorem
\ref{thm:original_case}. Note that we reach the stronger
conclusion of freeness over the scalars because the
entries of our matrix have the same distribution--depending
on whether they are diagonal or off diagonal. In the 
absence of this property we only get freeness over 
diagonal scalar matrices. 
\end{remark}

\section{Preliminaries and Notations}

\subsection{Operator valued probability spaces}
For a non commutative probability space $(\cA,\tau)$, let us
consider the non commutative probability space
$(M_d(\cA),\phi)$ where $\phi=\tr\otimes\tau$. That is $A
\in M_d(\cA)$,
$\phi(A)=\frac{1}{d}(\tau(A_{11})+\cdots+\tau(A_{dd}))$.

Let us denote by $\mathcal{D} = \mathcal{D}(\mathbb{C})
\subseteq \mathcal{D}(\cA) \subseteq M_d(\cA)$, the
subalgebras of scalar diagonal matrices and diagonal
matrices with entries in $\cA$.  We shall denote by
$\tilde\phi: M_d(\cA)\to \mathcal{D}$ and $\E:
M_d(\cA)\to\mathcal{D}(\cA)$, the unique conditional
expectations to $\mathcal{D}$ and, respectively, to
$\mathcal{D}(\cA)$ which are consistent with $\phi$.  For
$X=(x_{ij})_{i,j}\in M_d(\cA)$ these are given explicitly by
{\small\[
\E(X) =
\begin{pmatrix}  x_{11} & 0      & \cdots &   0    \\
                         0    & x_{22} & \cdots &   0    \\
                       \vdots & \vdots &        & \vdots \\
                         0    &   0    & \cdots &  x_{dd}
       \end{pmatrix},\
  \tilde \phi(X) =
\begin{pmatrix}  \tau(x_{11}) & 0      & \cdots &   0    \\
                         0    &\tau(x_{22}) & \cdots &   0    \\
                       \vdots & \vdots &        & \vdots \\
                         0    &   0    & \cdots &  \tau(x_{dd}).
       \end{pmatrix}.
\]}

Let us recall the notion of operator-valued probability
space and freenees with amalgamation (for a detailed
exposition see \cite{Sp98}).

\begin{definition}
A $\mathcal{B}$-valued probability space is a triplet $(\cA,
\mathcal{B}, \mathbb{E})$, consisting of a unital algebra
$A$, a unital subalgebra $\cB\subset \cA$ and a conditional
expectation $\mathbb{E}: \cA \rightarrow \mathcal{B}$ , i.e. a
unit-preserving linear map such that
$\mathbb{E}(b_1ab_2)=b_1\mathbb{E}(a)b_2$ for any
$a\in\mathcal{A}$ and $b_1, b_2 \in \mathcal{B}$.
\end{definition}

\begin{definition}
Given a $\mathcal{B}$ valued probability space $(\cA,
\mathcal{B}, \mathbb{E})$), a family of subalgebras
$(\cA_i)_i$ with $\mathcal{B} \subset \cA_i\subset \cA$ for
each $i$ is said to be free with amalgamation over
$\mathcal{B}$, if $\mathbb{E}(x_1 \cdots x_p) = 0$, whenever
$x_j \in\cA_{i(j)}$ , $\mathbb{E}(x_j ) = 0$, for all $j$,
and $i(j) \neq i(j + 1)$, $j = 1, \dots, p - 1$. A family
$\{s_1, \dots , s_r\}$ of $\mathcal{B}$-valued random
variables in $\mathcal{A}$ are free with amalgamation over
$\cB$, if the family of subalgebras $\mathrm{alg}\langle
s_i, \cB\rangle$, $i = 1, \dots , r$, are free with
amalgamation over $\mathcal{B}$.
\end{definition}

\subsection{Operator-valued free cumulants}
Let us denote by $NC(n)$ the set of non-crossing partitions
of of $[n]$ (\cite[chapter 9]{ns},  and by
$\mathcal{NC}=\cup^{\infty}_{n=1} NC(n)$.

For $n\in\mathbb{N}$, a $\mathbb{C}$-multi-linear map
$f:\mathcal{A}^{n}\to\mathcal{B}$ is called
$\mathcal{B}$\textit{-balanced} if it satisfies the
$\mathcal{B}$-bilinearity conditions, that for all
$b,b'\in\mathcal{B}$, $a_{1},\dots,a_{n}\in\mathcal{A}$, and
for all $r=1,\dots,n-1,$
\begin{eqnarray*}
f\left(ba_{1},\dots,a_{n}b'\right) &=&
bf\left(a_{1},\dots,a_{n}\right)b'\\ f\left(a_{1},\dots,a_{r}b,a_{r+1},\dots,a_{n}\right)
&=& f\left(a_{1},\dots,a_{r},ba_{r+1}\dots,a_{n}\right)
\end{eqnarray*}

A collection of $\mathcal{B}$-balanced maps
$\left(f_{\pi}\right)_{\pi\in \mathcal{NC}}$ is said to be
{multiplicative} with respect to the lattice of non-crossing
partitions if, for every $\pi\in \mathcal{NC}$, $f_{\pi}$ is
computed using the block structure of $\pi$ in the following
way:

1. If $\pi=\hat{1}_{n}\in NC\left(n\right)$, we just write
$f_{n}:=f_{\pi}$.

2. If $\hat{1}_{n}\neq\pi=\left\{ V_{1},\dots,V_{k}\right\}
\in NC\left(n\right),$ then by a known characterization of
$\mathcal{NC}$, there exists a block $V_{r}=\left\{
s+1,\dots,s+l\right\} $ containing consecutive elements. For
any such a block we must have
\begin{equation*}
  f_{\pi}\left(a_{1},\dots,a_{n}\right) = f_{\pi\backslash
    V_{r}}\left(a_{1}, \dots, a_{s} f_{l} \left(a_{s+1},
  \dots, a_{s+l}\right), a_{s+l+1}, \dots, a_{n}\right),
\end{equation*}
where $\pi\backslash V_{r}\in NC\left(n-l\right)$ is the
partition obtained from removing the block $V_{r}$.

The {operator-valued of $\cB$ free cumulants}
$\left(\kappa^{\cB}_{\pi}\right)_{\pi\in \mathcal{NC}}$ are
defined as the unique multiplicative family of
$\mathcal{B}$-balanced maps satisfying the (operator-valued)
moment-cumulant formulas
\begin{equation*}
\EE\left(a_{1}\dots a_{n}\right)=\sum_{\pi\in NC
  \left(n\right)}\kappa^{\cB}_{\pi}\left(a_{1},\dots,a_{n}\right)
\end{equation*}

By the free cumulants of a tuple
$(a_{1},\dots,a_{k})\in\mathcal{A}^k$, we mean the
collection of all cumulant maps
\begin{equation*}
\begin{array}{cccc}
\kappa_{i_{1},\dots,i_{n}}^{\mathcal{B};a_{1},\dots,a_{k}}:
& \mathcal{B}^{n-1} & \to & \mathcal{B},\\ &
\left(b_{1},\dots,b_{n-1}\right) & \mapsto &
\kappa^{\cB}_{n}\left(a_{i_{1}},b_{1}a_{i_2},\dots,b_{i_{n-1}}a_{i_n}\right)\end{array}
\end{equation*}
for $n\in\mathbb{N}$, $1\leq i_{1},\dots,i_{n}\leq k$.

Given subalgebras $(\cA_i)_i$ such that $\cB \subseteq \cA_i
\subseteq \cA$ for each $i$ and elements $a_1, \dots, a_n$
such that $a_j \in \cA_{i_j}$, a free cumulant map
$\kappa_{i_{1},\dots,i_{n}}^{\mathcal{B};a_{1},\dots,a_{k}}$
is {mixed} if there exists $r < s$ such that $i_{r}\ne
i_{s}$. The main feature of the operator-valued cumulants is
that they characterize freeness with amalgamation.

\begin{proposition}[\cite{Sp98}]
The random variables $a_{1},\dots,a_{n}$ are
$\mathcal{B}$-free if and only if all their mixed cumulants
vanish.
\end{proposition}

Let us finally state the formula for products as arguments,
which will be used in the proof of our main results.

\begin{proposition} \cite{Sp00} 
Suppose $n_1,\dots,n_r$ are positive integers and $n= n_1 +
\cdots + n_r$.  Given a $\cB$-valued probability space
$(\cA, \mathcal{B}, \mathbb{E})$ and
$
a_1, \dots , \ab a_{n_1},\ab  a_{n_1+1}, \dots, a_{n_1+n_2}, \dots,
a_{n_1 + \cdots + n_{r}} \in \mathcal{A}
$,
let $A_1 = a_1 \cdots a_{n_1}$, $A_2= a_{n_1+1} \cdots \ab
a_{n_1+n_2}$, \dots, $\ab A_{r} = a_{n_1 + \cdots +
  n_{r-1}+1} \cdots \ab a_{n_1 + \cdots + n_{r}}$. Then
\begin{equation} \label{arguments}
\kappa_{r}(A_1,A_2 \dots, A_{r-1},A_{r}) = \mathop{\sum_{\pi
    \in NC(n)}}_{\pi \vee \sigma = 1_{n}} \kappa_{\pi} (a_1,
\dots ,a_{n}),
\end{equation}
where $\sigma=\{\{1,2,\dots,n_1\} \cdots\{n_1+n_2 + \dots +
n_{r-1}+1,\dots,n_1+n_2 + \cdots + n_{r}\}\}$.
\end{proposition}

\section{Diagonal decompositions of $R$-cyclic matrices}

\subsection{Diagonal Decompositions}
\begin{notation}
Let $A \in M_d(\cA)$ be a matrix. We shall write $A = A_0 +
A_1 + \cdots + A_{d-1}$ where the $A_i$'s are shown in
Figure \ref{fig:terms_1_2_g}.

\begin{figure}[h]
\[
A_0 = \begin{pmatrix}  a_{11} & 0      & \cdots &   0    \\
                         0    & a_{22} & \cdots &   0    \\
                       \vdots & \vdots &        & \vdots \\
                         0    &   0    & \cdots &  a_{dd}
       \end{pmatrix}
A_1 = \begin{pmatrix}    0    & a_{12} &    0    & \cdots &   0    \\
                         0    &   0    & a_{23}  & \cdots &   0    \\
                       \vdots & \vdots &         &        & \vdots \\
                         0    &   0    &         &    0   & a_{d-1,d} \\
                       a_{d1} &   0    &         &    0   &    0  
       \end{pmatrix}
\] \medskip
\[
A_k = \begin{pmatrix}
0            &   \cdots    & & a_{1,k+1} &      0      &   \cdots      &     0         \\
0            &   \cdots    & &    0 \rule[-.3\baselineskip]{0pt}{1.6em}   &  a_{2,k+2} &   \cdots      &     0         \\
\vdots       &             & & \vdots &  \vdots     &  \ddots       &   \vdots      \\
0            &             & &        &             &               & a_{d-k,d}  \\
a_{d-k+1,1}  &            & &        &             &   \cdots      &     0         \\
             & \ddots      & &       &             &   \cdots      &     \vdots         \\
0            &     0       & a_{d,k} & \cdots &             &               &       0        \\
\end{pmatrix}
\]
\caption{\label{fig:terms_1_2_g} The first two terms in the
  diagonal decomposition of $A$ (top row), and the $k^{th}$
  term (second row).}
\end{figure}
We call this the \textit{diagonal decomposition} of $A$. We
interpret all subscripts modulo $d$, i.e. $A_k = A_l$ if $k
\equiv l$ $\pmod d$. The same assumption applies to the
indices of our matrices, namely $a_{ij} = a_{kl}$ whenever
$i \equiv k \pmod d$ and $j \equiv l \pmod d$. Since $d$
will remain fixed throughout this paper we shall write $i
\equiv j$ to mean $i \equiv j \pmod d$.
\end{notation}

\begin{notation} \label{def:xs-decomposition}.
Let $S$ be the matrix that cyclically permutes (backwards)
the standard basis of $\bC^d$,
\[
S = \begin{pmatrix}
0      &    1   &    0   & 0 & \cdots & 0 \\
0      &    0   &    1   &    & \cdots & 0 \\
\vdots & \vdots &  \ddots & \ddots  & \vrule width0pt depth 1em height 1em & \vdots \\
0      &    0   & \cdots &   &   0    &   1  \\
1      &    0   & \cdots &   &   0    &   0  \\
\end{pmatrix},
\]
and for $\omega = \exp(2\pi i/d)$, let
\[
\bD = \begin{pmatrix}
    1  &    0   &    0     &  \cdots & 0\\
    0  & \omega &    0     &  \cdots & 0\\
    0  &    0   & \omega^2 &  \cdots & 0 \\
\vdots & \vdots & \vdots   &  \vdots & \vdots \\
    0  &    0   &    0     &    0    & \omega^{d-1} \\
\end{pmatrix} .   
 \]
\end{notation}

Notice that the conditional $\E$ the expectation onto the diagonal
matrices with entries from $\cA$ may be written as
\begin{align}\label{eq:expectation}
\E(X) &= d^{-1}(X + \bD X\bD ^{-1} + \cdots + \bD ^{d-1}X \bD ^{-(d-1)}), 
\end{align}
and we have the commutation relations
\begin{align}
S^k\bD ^l &= \omega^{kl} \bD ^l S^k  \mbox{\ and\ } \bD ^l S^k = \omega^{-kl}  S^k\bD ^l. \notag
\end{align}

\begin{lemma}\label{lemma:matrix_entries}
Let $X =(x_{ij})$ be a $d \times d$ matrix. Then
$(XS)_{ij} = x_{i, j-1}$, $(SX)_{ij} = x_{i+1, j}$,
  $(XS^{-1})_{ij} = x_{i, j+1}$, and $(S^{-1}X)_{ij} =
  x_{i-1, j}$.  $\E(XS^k) = \E(S^{-k}X^t)$.
\end{lemma}

\begin{proof}
The first four equalities follow from the fact that $S_{ij}
= 1$ when $j \equiv i +1$ and $0$ otherwise. The last
equality follows from the fact that for any matrix $X$ we
have $\E(X^t) = \E(X)$ and $S^{-1} = S^t$.\
\end{proof}

Let $X = X_0 + X_1 + \cdots + X_{d-1}$ be the diagonal
decomposition of $X$.  In matrix notation $(X_k)_{ij} =
x_{ij}$ if $j \equiv i + k$ and $0$ otherwise. Recall that
here $i \equiv j$ means equivalent modulo $d$, the size of
the matrices.

\begin{lemma}\label{lemma:3}
$X_k = \E(XS^{-k})S^k = S^k \E(S^{-k}X)$.
\end{lemma}

\begin{proof}
\begin{align*}
(\E(XS^{-k})S^k)_{ij} &=
\begin{cases}
\E(XS^{-k}))_{ii} & j \equiv i + k \\
0 & j \not\equiv i + k
\end{cases} \\
&=
\begin{cases}
x_{i, i+k} & j \equiv i + k \\
0 & j \not\equiv i + k
\end{cases} \\
& = (X_k)_{ij},
\end{align*}
and
\begin{align*}
(S^k\E(S^{-k}X))_{ij} &=
\begin{cases}
\E(S^{-k}X))_{jj} & j \equiv i + k \\
0 & j \not\equiv i + k
\end{cases} \\
&=
\begin{cases}
x_{j-k, j} & i \equiv j - k \\
0 & j \not\equiv i + k
\end{cases} \\
& = (X_k)_{ij}.
\end{align*}
\end{proof}

\begin{notation}
Let $X \in M_d(\cA)$ and $Y_0, \dots, Y_{d-1}$ be the
\textit{diagonal decomposition} of $X^t$. Then by Lemmas
\ref{lemma:matrix_entries} and \ref{lemma:3}, $Y_k =
\E(X^tS^{-k})S^k = \E(S^k X)S^k$.
\end{notation}

\begin{lemma}\label{lemma:x-relation}
Let $X \in M_d(\cA)$. \\
i) If $X_0, \dots, X_{d-1}$ be the diagonal decomposition of $X$, then
\[
X_k = \frac{1}{d} \sum_{i=1}^d \omega^{ik} \bD ^i X \bD ^{-i}.
\]
ii) If $Y_0, \dots, Y_{d-1}$ be the
diagonal decomposition of $X^t$, then
\[
Y_k = 
S^k \bigg[\frac{1}{d} \sum_{i=1}^d \omega^{-ik} \bD ^i X \bD ^{-i}\bigg] S^k .
\]
\end{lemma}
\begin{proof}
For $(i)$ we use Eq. (\ref{eq:expectation}) and Lemma \ref{lemma:3}
\begin{multline*} 
X_k  = S^k \E(S^{-k}X) =
S^k\frac{1}{d} \sum_{i=1}^d \bD ^i(S^{-k} X )\bD ^{-i}
= 
\frac{1}{d} \sum_{i=1}^d \omega^{ik}  \bD ^i X  \bD ^{-i}\\
\end{multline*}
Similarly for  $(ii)$,
\begin{multline*} 
Y_k  = \E(S^k X)S^k =
\frac{1}{d} \sum_{i=1}^d \bD ^i(S^k X )\bD ^{-i} S^k \\
= 
\frac{1}{d} \sum_{i=1}^d \omega^{-ik} S^k \bD ^i X  \bD ^{-i} S^k 
= 
S^k \bigg[\frac{1}{d} \sum_{i=1}^d \omega^{-ik} \bD ^i X \bD ^{-i}\bigg] S^k.
\end{multline*}
\end{proof}

\subsection{Freeness over  $M_d(\bC)$ }

From now on we will assume that $X \in M_d(\cA)$, is free
from $M_d(\bC)$ over $\mathcal{D}$. The free cumulants
$\tilde\kappa_n:M_d(\cA)\to \mathcal{D}$ refer to operator
valued free cumulants over the algebra $\mathcal{D}$,
i.e. free cumulants with respect to $\tilde{\phi}.$

The following simple observation will be the beginning of our analysis.

\begin{lemma}\label{lemma:freeness S X}
If $X \in M_d(\cA)$, is free from $M_d(\bC)$ over
$\mathcal{D}$. Then $S$ is free from the family
$\{X_i\}^d_{i=1}$ over $\mathcal{D} $.
\end{lemma}

\begin{proof}
By Lemma \ref{lemma:x-relation}, for all $i$, $X_i$ is in
the algebra generated by $X$ and $\mathcal{D}$ and thus
since and $X$ is free from $M_d(\bC)$ over $\mathcal{D}$,
then the $\{X_i\}^d_{i=1}$ is also free from $S\in M_d(\bC)$
over $\mathcal{D}$.
\end{proof}

\begin{lemma} \label{(off diagonal)}
Let $M_0, M_1,M_2, \dots, M_{d-1}$ be the diagonal
decomposition of $M \in M_d(\mathcal{A})$.

Then for $i_1, \dots, i_r \in \{0, 1, \dots, d-1\}$ and
diagonal matrices $D_1,\dots,\ab D_r$, then
\[
\tilde\kappa_r(M_{i_1}D_1, \dots, M_{i_r}D_r)=0, \]
whenever  $ i_1 + \cdots + i_r \not\equiv 0.$
\end{lemma}

\begin{proof}

This follows from cumulant moment formula. Indeed, let $\mu$
be the M\"obius function for $NC(r)$ (see
\cite[Lect. 9]{ns}) then
\[
\tilde\kappa_r(M_{i_1}D_1, \dots, M_{i_r}D_r)= \kern-1.0em
\sum_{\pi\in NC(r)} \kern-0.8em
\mu(0_n, \pi)(\tilde\phi)_{\pi}(M_{i_1}D_1,M_{i_2}D_2,\cdots,M_{i_r}D_r).\]
Suppose $i_1+i_2+\cdots + i_r \not\equiv 0$. Then
for each partition $\pi$ there is at least one block
$V=\{b_1, \dots ,b_s\}$ such that
$i_{b_1}+i_{b_2}+\cdots+i_{b_s} \not\equiv 0$ and thus
$\tilde\phi(M_{i_{b_1}}D_{b_1}\cdots M_{i_{b_s}}D_{b_s})=0$.
\end{proof}

\begin{corollary}  \label{lemma:Y-cumulants}
Let $X \in M_d(\mathcal{A})$  and let 

\[\tilde X_i =X_{d-i}= \frac{1}{d} \sum_{l=1}^d \omega^{-il} \bD^l X \bD^{-l}. 
\]
Then
for $i_1, \dots, i_r \in [n]$ we have for $D_1, \dots, D_r \in \cD$
\[
\tilde\kappa_r(\tilde X_{i_1}D_1, \dots, \tilde X_{i_r}D_r) =0
\]
whenever  $ i_1 + \cdots + i_r \not\equiv 0 \pmod n.$

\end{corollary}
\begin{proof}
Indeed if $ i_1 + \cdots + i_r \not \equiv 0$ then $d- i_1 + \cdots +d- i_r \not \equiv 0$ and hence
$\tilde\kappa_r(\tilde X_{i_1}D_1, \dots, \tilde X_{i_r}D_r)=0 =\tilde\kappa_r(X_{d-i_1}D_1, \dots, X_{d-i_r}D_r)=0 $ by Lemma
\ref{(off diagonal)}.
\end{proof}

\begin{lemma}\label{lemma:S-cumulants}
Let $S$ be as above.  Then for all $D_1,.\dots,D_{2r}\in \mathcal{D}$.

\medskip\noindent
$i)$
$\tilde\kappa_r(S^{i_1}D_1,S^{i_2}D_2,\cdots,S^{i_r}D_r)=0$
unless $i_1+i_2+\cdots + i_r \equiv 0 $

\medskip\noindent
$ii)$
$\tilde\kappa_{2r}(S^{i_1}D_1,S^{-i_1}D_2,\cdots,S^{i_r}D_{2r-1},S^{-i_r}D_{2r})=0$,
and \\ $\tilde\kappa_{2r}(S^{-i_r}D_{2r},S^{i_1}D_1,S^{-i_1}D_2,\cdots,S^{i_r}D_{2r-1})=0$
unless $i_1 \equiv i_2 \equiv \cdots \equiv i_r$.
\end{lemma}

\begin{proof}

$(i)$ Follows from Lemma \ref{(off diagonal)}, since $S^i=\mathbb{J}_i$ in the diagonal decomposition of the matrix $\mathbb{J}=(1)^d_{ij}$, i.e. all entries equal to $1$.

$(ii)$ We use induction on $r$. First we prove that
  $\tilde\kappa_4(S^kD_1,S^{-k}D_2, \ab S^l D_3,S^{-l} D_4)
  \ab = 0$ and  $\tilde\kappa_4(S^{-l} D_4,S^kD_1,S^{-k}D_2, \ab S^l D_3,)
  \ab = 0$ unless $l \equiv k \pmod d$. This will be the
  base of our induction. 

Indeed suppose $k \not\equiv l$ and
  $k \not\equiv 0 \pmod d$, by hypothesis there exists $D\in
  \mathcal D$ such $D=S^kD_1S^{-k}D_2$, by the formula for
  products as arguments \eqref{arguments} we have
$$0=\tilde\kappa_3(D,S^lD_3,S^{-l}D_4)=
\mathop{\sum_{\pi \in NC(4)}}_{\pi\vee\sigma = 1_4}
\tilde\kappa_\pi(S^kD_1,S^{-k}D_2,S^lD_3,S^{-l}D_4),$$
where $\sigma=\{\{1,2\}\{3\}\{4\}\}.$ 
Thus,
\begin{multline*}
0 =
\tilde\kappa_4(S^kD_1,S^{-k}D_2,S^l D_3,S^{-l} D_4) \\
\mbox{} +\tilde\kappa_2(S^kD_1\tilde\kappa_2(S^{-k}D_1,S^{l}D_3), S^{-l}D_4) \\ \mbox{} +\tilde\kappa_3(S^kD_1,\tilde\kappa_1(S^{-k}D_2)S^lD_3,S^{-l}D_4)\\
\mbox{} +
\tilde\kappa_1(S^kD_1)\tilde\kappa_3(S^{-k}D_2,S^lD_3,S^{-l}D_4).
\end{multline*}

By (\textit{i}), all of the terms in the second and third
line equal $0$, and thus $\tilde\kappa_4(S^kD_1,S^{-k}D_2,S^l D_3,S^{-l} D_4)=0.$
 
Similarly,
$$0=\tilde\kappa_3(S^{-l}D_4, D,S^lD_3)=
\mathop{\sum_{\pi \in NC(4)}}_{\pi\vee\rho = 1_4}
\tilde\kappa_\pi(S^{-l}D_4, S^kD_1,S^{-k}D_2,S^lD_3)$$
where $\rho=\{\{1\},\{2,3\}\{4\}\}.$ 
Thus,
\begin{multline*}
0 =
\tilde\kappa_4(S^{-l} D_4,S^kD_1,S^{-k}D_2,S^l D_3) \\
\mbox{} +\tilde\kappa_2(S^{-l} D_4,S^kD_1\tilde\kappa_2(S^{-k}D_1,S^{l}D_3)) \\ \mbox{} +\tilde\kappa_3(S^kD_1,\tilde\kappa_1(S^{-l} D_4,S^{-k}D_2)S^lD_3)\\
\mbox{} +
\tilde\tilde\kappa_3(S^{-l}D_4\kappa_1(S^kD_1),S^{-k}D_2,S^lD_3).
\end{multline*}

Again, by (\textit{i}), all of the terms in the second and third
line equal $0$, and thus $\tilde\kappa_4(S^{-l} D_4,S^kD_1,S^{-k}D_2,S^l D_3)=0,$ which finishes the base of induction.

Now, we assume that (\textit{ii}) is true for all $1\leq
t\leq r$ and prove that it also holds for $r+1$.  That is,
we will prove that
\[
\tilde\kappa_{2r+2}(S^kD_1,S^{-k}D_2,
S^{i_1}D_3,S^{-i_1}D_4,\cdots,S^{i_r}D_{2r},S^{-i_r}D_{2r+2})=
0,
\]
unless $k \equiv i_1 \equiv i_2 \equiv \cdots \equiv i_r$. 

Again we write $D=S^kD_1S^{-k}D_2$, and use the formula for
products as arguments \eqref{arguments}, yielding
\begin{multline*}
0=\tilde\kappa_{2r+1}(D, S^{i_1}D_3,S^{-i_1}D_4,\cdots,S^{i_r}D_{2r+1},S^{-i_r}D_{2r+2})\\ 
= \kern-1em
\mathop{\sum_{\pi\in NC(2r+2)}}_{\pi\vee\sigma = 1_{2 r + 2}} \kern-1em
 \tilde\kappa_{\pi}(S^kD_1,S^{-k}D_2, S^{i_1}D_3,S^{-i_1}D_4,\cdots,S^{i_r}D_{2r+1},S^{-i_r}D_{2r+2})
\end{multline*} 
where $\sigma=\{\{1,2\},\{3\},\{4\},\dots,\{2r+2\}\}.$ If we
consider the partitions $\pi \in NC(2 r + 2)$ such that $\pi
\vee \sigma = 1_{2 r + 2}$ we have, apart from $1_{2 r +
  2}$, the collection $\pi_2, \dots , \pi_{2 r + 2}$ where
$\pi_i= \{ \{1, i+1, \dots, 2 r + 2\}, \{2, 3, \dots,
i\}\}$.  This is because $\pi \vee \sigma = 1_{2 r + 2}$
implies that $\pi$ has at most two blocks: one containing the element 1
and one containing the element 2. 

If  $i=2j$, then 
\begin{multline*}
\tilde\kappa_{\pi_{2j}}(S^kD_1, S^{-k}D_2, \dots,
S^{i_r}D_{2r-1}, S^{-i_r}D_{2r}) \\ =
\tilde\kappa_{2r-2j+1}(S^kD_1\Delta_j,S^{-i_j}D_{2j+1},\cdots,
S^{i_r}D_{2r+1},S^{-i_r}D_{2r+2}),
\end{multline*}
where $\Delta_{2j-1}=\tilde\kappa_{2j-1}(S^{-k}D_2,
S^{i_1}D_{3},S^{-i_1}D_4,\cdots,S^{i_j}
D_{2j})\in\mathcal{D}.$ Since $k\neq0$, then by (i), $\Delta_{2j-1}=0$. 

Second, if $j$ is odd, say $j=2j+1$ then 
\begin{multline*}
\tilde\kappa_{\pi_{2j+1}}(S^kD_1, S^{-k}D_2, \dots,
S^{i_r}D_{2r-1}, S^{-i_r}D_{2r}) \\ =
\tilde\kappa_{2r-2j}(S^kD_1\Delta_j,S^{-i_j}D_{2j+2},\cdots,
S^{i_r}D_{2r+1},S^{-i_r}D_{2r+2}),
\end{multline*}
where $\Delta_{2j}=\tilde\kappa_{2j}(S^{-k}D_2,
S^{i_1}D_{3},S^{-i_1}D_4,\cdots,S^{i_j}
D_{2j+1})\in\mathcal{D}.$

If it not the case that $k \equiv i_1 \equiv i_2 \equiv
\cdots \equiv i_r$, then the same condition will hold for at
least one block of $\pi_j$. Hence
$\tilde\kappa_{\pi_i}(S^kD_1, S^{-k}D_2, \dots,\ab
S^{i_r}D_{2r-1}, S^{-i_r}D_{2r})= 0$. This proves that the
cumulant $\tilde\kappa_{2r+2}(S^kD_1, \ab S^{-k}D_2,
S^{i_1}D_3,S^{-i_1}D_4,\cdots,S^{i_r}D_{2r},S^{-i_r}D_{2r+2})=
0$, as desired.
\end{proof}

\begin{definition}
Let $y_1, y_{-1} \in (\cA, \phi)$ a non-commutative
probability space. If for all $n$ and all $\epsilon_1,
\epsilon_2, \dots, \epsilon_n \in \{-1, 1\}$ we have
$\tilde\kappa_n( y_{\epsilon_1}, y_{\epsilon_2},\ab \dots ,
y_{\epsilon_n}) = 0$ unless $n$ is even and $\epsilon_i =
-\epsilon_{i+1}$ for $i =1, \dots, n-1$, we say that $\{y_1,
y_{-1} \}$ are a $R$-\textit{diagonal pair}, following 
\cite[Thm. 3.1]{ss01} (see also \cite{bd}).
\end{definition}

\begin{remark}
If $(\cA, \phi)$ is a $*$-probability space, then $y$ is
$R$-diagonal if and only if $\{y, y^*\}$ is an $R$-diagonal
pair.
\end{remark}

\begin{theorem}\label{main 1}
Suppose $X$ is free from $M_d(\bC)$ over $\cD$. Let $Y_0$, $Y_1$,
\dots, $Y_{d-1}$ be the diagonal decomposition of $X^t$.

For $d$ even, $Y_0, \{ Y_1, Y_{d-1}\}, \{Y_2, Y_{d-2}\},\ab
\dots, \{Y_{d/2-1}, Y_{d/2 + 1}\}, Y_{d/2}$ is a free family
over $\cD$ and $\{Y_i, Y_{d-i}\}$ is a $R$-diagonal pair for $i = 1,
\dots,\ab d/2-1$.

For $d$ odd, $Y_0, \{ Y_1, Y_{d-1}\}, \{Y_2, Y_{d-2}\},
\dots, \{Y_{(d-1)/2}, Y_{(d+1)/2}\}, $ is a free family over $\cD$ and
$\{Y_i, Y_{d-i}\}$ is a $R$-diagonal pair for $i = 1, \dots,
(d-1)/2$.
\end{theorem}

\begin{proof}
We write $Y_i = S^i \tilde{X}_i S^i$ for $i = 0, \dots, d -1$ with
$\tilde X_i$ as in Lemma \ref{lemma:x-relation}. Let $r > 1$ and
$i_1, \dots, i_r \in \{0, 1, \dots, d-1\}$ be given. We must
show that $$\tilde\kappa_r( Y_{i_1}D_1, \dots, Y_{i_r}D_r) = 0$$
unless either $i_1 = \cdots = i_r = 0$, or $r$ is even and
$i_l + i_{l+1} = d$ for $l=1 , \dots, r - 1$.

Suppose that for some $l_1, \dots ,l_k$ we have $i_{l_1} =
\cdots = i_{l_k} = 0$ and for all other $j$'s we have $i_j
\not = 0$. Let $s_j = |\{ m \mid 1 \leq m < j$ and $i_m = 0
\}|$ and $s = |\{ m \mid 1 \leq m \leq r$ and $i_m = 0
\}|$. Let $\sigma = \{ V_1, \dots, V_r\}$ be the interval
partition with $r$ blocks constructed as follows. $V_j = \{
3 j - 2 s_j -2\}$ if $i_j =0$ and $V_j = \{ 3 j - 2 s_j -2,
3 j - 2 s_j -1, 3 j - 2 s_j \}$ if $i_j \geq 1$. In other words, for each $k$ such that $ i_{l_k} = 0$,  $\sigma$ has a block of size $1$ and for each $k$ such that $ i_{l_k} \neq0$,  $\sigma$ has a block of size $3$. For
example, if $(i_1, i_2, i_3, i_4, i_5) = (0, 2, 0, 1, 2)$ we
have $\sigma = \{(1),(2,3,4),(5),(6,7,8),(9, 10, 11)\}$.

For $ 1 \leq m \leq 3r - 2s$, let us define $Z_m$ as
follows. If $m \in V_j$ and $i_j = 0$ then $Z_m = X_0D_j$. If
$m \in V_j$ and $i_j \geq 1$ then $Z_m = \tilde X_{i_j}$ if $m = 3
j - 2 s_j -1$ , $Z_m = S^{i_j}$ if $m =3 j -2 s_j -2$ and 
$Z_m = S^{i_j}D_j$ if $m = 3 j - 2 s_j $. In the example above $Z_1, \dots,
Z_{11}$ are $X_0D_1, S^2, X_2, S^2D_2, X_0D_3, S^1, X_1, S^1D_4,S^2, X_2, S^2D_5$ respectively. Then, by the formula for products as arguments \eqref{arguments}, we have
\[
\tilde\kappa_r(Y_{i_1}D_1, \dots, Y_{i_r}D_r)
=
\mathop{\sum_{\pi \in NC(3r-2s)}}_{\pi \vee \sigma = 1_{3r-2s}}
\tilde\kappa_\pi(Z_1, \dots, Z_{3r-2s} )
\]

The proof will consist in analyzing the non-vanishing terms in the formula above. Observe that by Lemma \ref{lemma:freeness S X}, $S$ and the $X_i$'s are free. Thus for $\tilde\kappa_\pi(Z_1, \dots, Z_{3r-2s} ) \not = 0$
we must have the blocks of either $\pi$ to consist of
$X$-blocks (only connecting $X$'s and $XD$'s) and $S$-blocks (only
connecting $S$'s and $SD$'s).

\

We shall break the proof into two cases. Case 1 is
when for some $l \in [r]$ we have $i_l = 0$, i.e. $s>0$. Case 2
is when $i_l > 0$ for all $l \in [r]$,  i.e. $s=0$.

\

\noindent \textbf{Case 1.} Mixed cumulant with at least one $Y_0$ as an entry,  i.e. $s>0$. 

\

Let $W$ be an $X$-block and let $b \in W$ be such that
$Z_b = X_0D$.  If $W$ is a singleton then the condition  $\pi
\vee \sigma = 1_{3r - 2s}$ will not be satisfied. So $W$ must have another element.

Without loss of generality, there is at least one element to the right of $b$ ( the following argument also applies if there is an element to the left).
Now, let $b' \in W$ be the next element to the right of $W$; it
must be a $X_i$, by freeness. Thus we may label the elements between $b$
and $b'$ as below. By Lemma \ref{lemma:S-cumulants} (\textit{i}) we have $2
i_{t_1} + \cdots + 2 i_{t_{k-1}} + i_{t_k} \equiv 0 \pmod
d$. 
\begin{center}
\begin{tikzpicture}
\node  at (1,1) {$\cdots  \mid X_0 D_{t_1-1} \mid \ds\mathop{\underbrace{S^{i_{t_1}},\tilde X_{i_{t_1}},S^{i_{t_1}}D_{t_1}  \mid  \cdots  \mid S^{i_{t_k}}}}_{},\tilde X_{i_{t_k}},S^{i_{t_k}}D_{t_k}  \mid \cdots$};
\draw (-3.00,1) -- (-3.00,0.5);
\draw (-4.00,0.5) -- (4.75, 0.5);
\draw (3.5,1) -- (3.5,0.5);
\node at (-3.00, 0.2) {$b$};
\node at (3.5, 0.2) {$b'$};
\end{tikzpicture}
\end{center}
By Corollary \ref{lemma:Y-cumulants} we must have
$i_{t_1} + \cdots + i_{t_{k-1}} \equiv 0 \pmod d$. Hence
$i_{t_k} \equiv 0 \pmod d$. Moreover, the blocks of $\sigma$
between $b$ and $b'$ will not be joined by $\pi$ to the
other blocks of $\sigma$ and required by the condition $\pi
\vee \sigma = 1_{3r - 2s}$.
Thus there cannot be any blocks
of $\sigma$ between $b$ and $b'$, and so $b'=b+1$. 

 Finally, if $b$ and $b'$ are the only elements of $W$, then again the condition $\pi\vee \sigma = 1_{3r - 2s}$ will not be satisfied. Proceeding in the same way we may deduce that there is only one $X$-block, with entries of the form $X_0D$. Hence $i_1 =
i_2 = \cdots = i_r = 0$. 

\

\noindent \textbf{Case 2.} Mixed cumulant with no $Y_0$ as an entry.

\

In this case, where $i_l \geq 1$ for $1 \leq l
\leq r$. Let $\sigma \in NC(3r)$ be the partition $\{(1, 2,
3), (4,5, 6), \dots, (3r -2,\ab 3r -1 , 3r)\}$. Then
\[
\tilde\kappa_r( Y_{i_1}D_1, \dots, Y_{i_r}D_r)
=
\mathop{\sum_{\pi \in NC(3r)}}_{\pi \vee \sigma = 1_{3r}}
\tilde\kappa_\pi(S^{i_1}, \tilde X_{i_1}, S^{i_1}D_1, \dots, S^{i_r}, \tilde X_{i_r}, S^{i_r}D_r)
\]
Suppose $\pi \in NC(3r)$ and $\tilde\kappa_\pi(S^{i_1}, \tilde X_{i_1},
S^{i_1}D_1, \dots, S^{i_r}, \tilde X_{i_r}, S^{i_r}D_r) \not = 0$. Again,
since $S$ if free over $\cD$ from the $X_i$'s, the blocks of $\pi$ must
either be $S$-blocks, or $X$-blocks.

We now consider the case of $S$-blocks. Let $b = 3 s_1$ and
consider $W$, the $S$-block of $\pi$ and such $b \in
W$. Since $i_{s_1}\not\equiv 0$, $b$ must be connected to
some other element. Let $b' \in W$ be the next point to the
right ( if $b$ is the last
element in $W$, then $b'$ is the smallest element of $W$ and a similar argument applies). Then there are $s_1 < \cdots < s_k$ such
that either $b' = 3s_k-2$ or $b' = 3s_k$.

First, suppose that $b' = 3 s_k$.

\smash{\begin{tikzpicture}
\node  at (1,1) {$\cdots , \tilde X_{i_{s_1}} , \ds S^{i_{s_1}} D_{s_1}\mid \underbrace{S^{i_{s_2}}, \cdots, S^{i_{s_{k-1}}} D_{s_{k-1}}  \mid S^{i_{s_k}} D_{s_k} ,  \tilde X_{i_{s_k}} }, S^{i_{s_k}} D_{s_k} \mid \cdots$};
\draw (-2.7,0.7) -- (-2.7,0.3);
\draw (-4.25,0.3) -- (7, 0.3);
\draw (5.5,0.7) -- (5.5,0.3);
\node at (-2.7, 0.0) {$b$};
\node at (5.5, 0.0) {$b'$};
\end{tikzpicture}}

Then by Lemma \ref{lemma:S-cumulants} (\textit{i}), $
2i_{s_2} + \cdots + 2i_{s_{k-1}} + i_{s_{k}}
\equiv 0 \pmod d$. By Corollary \ref{lemma:Y-cumulants} we
must have $i_{s_1} + \cdots + i_{s_{k}} \equiv 0 \pmod
d$. Hence $ i_{s_{k}} \equiv 0 \pmod d$, but this
case has been ruled out already.

Now, suppose $b' = 3 s_k-2$. If there are
any blocks of $\sigma$ between $b$ and $b'$ they will not be
able to joined to the others by $\pi$; contradicting our
assumption that $\pi \vee \sigma = 1_{3r}$. 

\noindent\kern-1.25em\smash{\begin{tikzpicture}
\node  at (1,1) {$\cdots  \mid  S^{i_{s_1}}, \tilde X_{i_{s_1}} , \ds S^{i_{s_1}} D_{s_1}\mid S^{i_{s_2}},\tilde X_{i_{s_2}}, \cdots , S^{i_{s_{k-1}}} D_{s_{k-1}}  \mid S^{i_{s_k}} ,  \tilde X_{i_{s_k}}, S^{i_{s_k}} D_{s_k} \mid \cdots$};
\draw (-2.5,0.7) -- (-2.5,0.3);
\draw (-4.25,0.3) -- (6, 0.3);
\draw (3.75,0.7) -- (3.75,0.3);
\node at (-2.5, 0.0) {$b$};
\node at (3.75, 0.0) {$b'$};
\end{tikzpicture}}\\
Thus we must have that $k = 2$ and $s_2 = s_1 + 1$.

Thus, any $b=3 s_1$ must be connected to $3s_1+1$.  A similar argument shows that any $b = 3s_1 +1$ must be connected to $3s_1$.

We thus have arrived at the following form for any $S$-block:
$\{ 3b_{1},\ab 3 b_{1}+1, 3b_{2},3b_{2}+1 , \dots  3b_{l}+1\}$$  \pmod {3r}$.  Notice that in particular this means that $r$ is even.

By looking at the elements between $3b_i+1$ and $3b_{i+1}$,
using Lemma \ref{lemma:S-cumulants} (\textit{i}), and
Corollary \ref{lemma:Y-cumulants} we may conclude that
$i_{b_{r}+1}\equiv -i_{b_{r+1}}$, for all $i=1,\dots,l$.
This implies by Lemma \ref{lemma:S-cumulants} (\textit{ii})
that $i_{b_{1}}\equiv i_{b_{2}}\equiv\cdots \equiv
i_{b_{n}}$.

We say a block is of type $k$ if $i_{b_1} = k$ and $0 \leq k
\leq d$.  Notice that the block next to $W$ (a block $W$ such
that $3b_i-2$ or $3b_i+3$ belongs to $W'$ ,for some $1\leq i
\leq l$) must have type $-k$. Thus implies $i_{1}\equiv
-i_{2}\equiv\cdots \equiv i_{r}$ which proves both claims.
\end{proof}

\subsection{Scalar case}

In this subsection we consider the case where $X \in M_d(\cA)$ is $\phi$-free from
$M_d(\bC)$.  This corresponds to uniform $R$-cyclic matrices. These are matrices such that the non-vanishing cumulants of the entries depend only on the length.  The relation with free compression by matrix units may be seen in \cite[Theorem 14.20]{ns}. The importance of them in random matrix theory is that they appear as limits of unitarily invariant random matrices.

The main difference from the $\mathcal{D}$-valued case is the analog of Lemma \ref{lemma:freeness S X}, whose proof needs a little more work in the scalar case.

\begin{lemma}  \label{corollary:s_x_freeness}
Suppose $X \in M_d(\cA)$ is $\phi$-free from
$M_d(\bC)$. Then $S$ is free from the family
$\{X_i\}^d_{i=1}$.
\end{lemma}

In order to prove the lemma above we need a series of results regarding freeness and conjugation with powers of the matrix $S$.

\begin{lemma}\label{lemma:centred_polynomial}
Let $C$ be a centred polynomial in $S$. Then for all $k$ and $l$ we have $\phi(C\bD ^l) =
\phi(\bD ^{-k}C\bD ^l) = 0$; and for $l \not\equiv 0$ we
have $\phi(\bD ^l) = 0$
\end{lemma}

\begin{proof}
$\phi(\bD ^l) = d^{-1}(1 + \omega^l + \omega^{2l} + \cdots +
  \omega^{(d-1)l}) = 0$ for $l \not \equiv 0$. Also $S^m
  \bD _l$ and $\bD ^{-k}S^m \bD ^l$ are 0 on the diagonal for $m
  \not\equiv 0$. Thus $\phi(S^m \bD _l) =
  \phi(\bD ^{-k}S^m \bD ^l) = 0$.
\end{proof}

Let us recall from \cite[Lemma 3.8]{hla} the following lemma. 
\begin{lemma}\label{lemma:1}
Suppose $(\mathcal{A},\tau)$ is a non-commutative probability space. Let $C\in \mathcal{A}$ be $\tau$-free from the algebra
$\mathcal{B}$ and let $U \in \mathcal{B}$ be a unitary operator
with $U^d = 1$ and such that $\phi (U^k) = 0$ when $k
\not\equiv 0 \pmod d$. Let $C_i = U^i C U^{-i}$ for $i =1,
\dots,d$. Then $C_1$, $C_2$, \dots, $C_d$ are $\tau$-free.
\end{lemma}

\begin{corollary}
Suppose $X \in M_d(\cA)$ is $\phi$-free from
$M_d(\bC)$. Then the set $\{\bD X\bD ^{-1}, \bD ^2 X \bD
^{-2}, \dots, \bD ^d X \bD ^{-d} \}$ is a free family.
\end{corollary}

The following lemma is a generalization of Lemma \ref{lemma:freeness S X} where we now show
that $S$ is free from $\{\bD X\bD ^{-1}, \bD ^2 X \bD ^{-2},
\dots, \bD ^d X \bD ^{-d} \}$.

\begin{lemma}\label{lemma:s_x_freeness}
Suppose $X \in M_d(\cA)$ is $\phi$-free from
$M_d(\bC)$. Then the set $\{S, \bD X\bD ^{-1}, \bD ^2 X \bD
^{-2}, \dots, \bD ^d X \bD ^{-d} \}$ is a free family.
\end{lemma}

\begin{proof}
Let $\cB_0 = M_d(\bC)$ and $\cB_1 = \alg(1, X)$ be the
algebra generated by $1$ and $X$. By assumption $\cB_0$ and
$\cB_1$ are free, so any word which is an alternating
product of centred elements is centred. For $i = 1, \dots,
d$, let $\cA_i = \bD ^{i} \alg(1, X) \bD ^{-i}$ and $\cA_0 =
\alg(1, S)$. We must show that $\cA_0, \cA_1, \dots, \cA_d$
are free. Let $W$ be an alternating product of centred
elements in the algebras $\cA_0, \cA_1, \dots, \cA_d$. We
will show that $W$ is also an alternating product of centred
elements in the algebras $\cB_0$ and $\cB_1$ and thus is
centred. This will show that $\cA_0, \cA_1, \dots, \cA_d$
are free.

So, to this end let $W$ be an alternating product of centred
elements in the algebras $\cA_0, \cA_1, \dots, \cA_d$. We
may write $W = V_1 \cdots V_k$ with $V_i \in \cA_{j_i}$ such
that $j_1, \dots, j_k \in \{0,1, \dots, d\}$, $j_1 \not= j_2
\not= \cdots \not= j_k$ and $\phi(V_i) = 0$. For each $i$
for which $j_i > 0$ there is $A_i \in \alg(1, X)$ with
$\phi(A_i) = 0$ and $V_i = \bD ^{j_i} A_i \bD ^{-j_{i}}$.

If $i$ is such that $j_i \not = 0$ then we can write $V_i
V_{i+1}$ as either $\bD ^{j_i} A_i\ab \bD ^{-j_i + j_{i+1}} A_{i +
  1} \cdots $ if $j_{i+1} \not= 0$ or $\bD ^{j_i} A_i
[\bD ^{-j_i}V_{i+1}\bD ^{j_{i+2}}] A_{i+2} \cdots$ if $j_{i+1} =
0$ (with $\bD ^{-j_i}V_{i+1}\bD ^{j_{i+2}} \in \cB_0$). In either
case, by Lemma \ref{lemma:centred_polynomial}, $A_i$ is
followed by a centred element of $\cB_0$.  If $j_i = 0$ then
$j_{i+1} \not = 0$ so $V_{i-1}V_i V_{i+1} = \cdots A_{i-1}
[\bD ^{-{j_{i-1}}}V_i \bD ^{j_{i+1}}]A_{i+1} \cdots$. So again we
have a alternating product of centred elements of $\cB_0$
and $\cB_1$. This proves the claim.
\end{proof}

Now we are ready to prove Lemma \ref{corollary:s_x_freeness}.
\begin{proof}[Proof of Lemma $\ref{corollary:s_x_freeness}$]
We have by Lemma \ref{lemma:s_x_freeness}, $\{S, \bD X\bD ^{-1}, \bD ^2 X \bD ^{-2}, \ab
\dots, \ab \bD ^d X \bD ^{-d} \}$ is a free family. 
In particular, we have that $S$ is
free from any linear combination of $\{\bD X\bD ^{-1}$,\ab  $\bD
^2 X \bD ^{-2}$, \dots, $\bD ^d X \bD ^{-d}\}$ and thus from
$\{X_i\}_i$, by Lemma $\ref{lemma:x-relation}$.
\end{proof}

Next we come to the analogs of Lemmas \ref{(off diagonal)}, \ref{lemma:Y-cumulants}, and \ref{lemma:S-cumulants}. The reader will easily convince himself that the proofs of these lemmas are
also valid in the case $X \in M_d(\cA)$ is $\phi$-free
from $M_d(\bC)$, by replacing the diagonal matrices $D_i$'s
by scalars (or by $1$), and then use Lemma~\ref{corollary:s_x_freeness} instead of Lemma
\ref{lemma:freeness S X}.

Similarly the proof of the following theorem follows exactly the same steps as the proof of Theorem \ref{main 1}, with the obvious changes.

\begin{theorem}\label{theorem: main 1:scalar}
Suppose $X$ is free from $M_d(\bC)$ over $\bC$. Let $Y_0$, $Y_1$,
\dots, $Y_{d-1}$ be the diagonal decomposition of $X^t$.

For $d$ even, $Y_0, \{ Y_1, Y_{d-1}\}, \{Y_2, Y_{d-2}\},\ab
\dots, \{Y_{d/2-1}, Y_{d/2 + 1}\}, Y_{d/2}$ is a free family
over $\bC$ and $\{Y_i, Y_{d-i}\}$ is a $R$-diagonal pair for $i = 1,
\dots,\ab d/2-1$.

For $d$ odd, $Y_0, \{ Y_1, Y_{d-1}\}, \{Y_2, Y_{d-2}\},
\dots, \{Y_{(d-1)/2}, Y_{(d+1)/2}\}, $ is a free family over $\bC$ and
$\{Y_i, Y_{d-i}\}$ is a $R$-diagonal pair for $i = 1, \dots,
(d-1)/2$.
\end{theorem}

\section{Return to the Wishart case}\label{sec:return_wishart}

Let us conclude by returning to the case of a Wishart matrix
$W$ (c.f. Eq. (\ref{eq:wishart})) where $d_1$ is fixed and
$p/d_2 \rightarrow c d_1$, but in the case of arbitrary
$d_1$. Let $w$ be the limit distribution of the Wishart
matrix in Eq. (\ref{eq:wishart}). We write $d_1 w$ as a $d_1
\times d_1$ matrix: $d_1 w = (w_{ij})_{ij}$. Then $d_1
W^\ltr$ has the limit distribution $d_1w^t =
(w_{ji})_{ij}$. When $d_1$ is even let
\begin{equation}\label{eq:even_decomposition}
d_1 w^t = Y_0 + Y_1 + Y_{d_1 -1} + \cdots +
Y_{d_1/2 - 1} + Y_{d_1/2 + 1} + Y_{d_1/2},
\end{equation}
and when $d_1$ is odd let
\begin{equation}\label{eq:odd_decomposition}
d_1 w^t = Y_0 + Y_1 + Y_{d_1 -1} + \cdots +
Y_{(d_1 + 1)/2 - 1} + Y_{(d_1 + 1)/2 + 1}. 
\end{equation}
Since all the off-diagonal entries of $w$ are
$R$-diagonal with the same distribution:
$\tilde\kappa_{2n}(w_{12}, w_{21}, \dots, w_{12}, w_{21}) = c$ and
the diagonal entries are all $\mrp_{c}$, we have that 
$\tilde\kappa_r(Y_{i_1}, \dots, Y_{i_r}) = \kappa_r(Y_{i_1}, \dots, Y_{i_r})$.

Thus, we are in the scalar case, as in Section 3.3. Thus, we arrive to the next theorem that states that the diagonal decomposition of $d_1w^t$ is a free decomposition. 

\begin{theorem}\label{thm:original_case}
Let $w$ be the limit distribution of the Wishart matrix $W$ in Eq. $(\ref{eq:wishart})$. Suppose 
$d_1$ is fixed and $p/d_2 \rightarrow c d_1$. Then
 $Y_0$, $Y_1$, \dots, $Y_{d_1/2}$
are $*$-free when $d_1$ is even and when $d_1$ is odd we
have $Y_0$, $Y_1$, \dots, $Y_{(d_1+1)/2}$ are
$*$-free. Moreover $Y_0$ has the distribution $\mrp_{cd_1}$,
$Y_{d_1 - i} = Y_i^*$, and for $i \geq 1$, $Y_i$ is an
$R$-diagonal operator with $\kappa_{2n}(Y_i, Y^*_i, \dots,
Y_i, Y_i^*) = cd_1$. When $d_1$ is even $Y_{d_1/2}$ is even with
$\kappa_{2n}(Y_{d_1/2}, Y^*_{d_1/2}, \dots,
Y_{d_1/2}, Y_{d_1/2}^*) = cd_1$. 
\end{theorem}

\section*{Acknowledgements}
This work was carried out during the program `New
Developments in Free Probability and Applications' at the
Centre de Recherches Math\'ematiques (Montreal) March 1 -
31, 2019, as part of the 50$^{\mathrm{th}}$ anniversary year
program. The authors are very grateful to the director and
the staff of the Centre for the funding and hospitality.

\thebottomline
\end{document}